\documentclass[letterpaper,reqno,11pt]{amsart}

\usepackage{preamble}

\title{Fine Mixed Subdivisions of a Dilated Triangle}

\author{Yuan Yao}
\author{Fedir Yudin}
\email{\{yyao1,fedir\}@mit.edu}

\thanks{Both authors were supported by the MIT Math Department's Supervised UROP (UROP+).}

\begin{document}

\begin{abstract}
    An equilateral triangle of side $n$ can be tiled with $n$ unit equilateral triangles and $\frac{n(n-1)}{2}$ unit rhombi with $60^{\circ}$ and $120^{\circ}$ angles. In this paper, we focus on understanding such tilings with a fixed arrangement of unit triangles. We formulate a criterion for such a tiling being unique, identify a minimal set of operations that connects all such tilings, and determine which parts are common to all such tilings. Additionally, we establish a nearly tight bound on the number of ``trapezoid flips'' needed to connect any two tilings with different arrangements of triangles.
\end{abstract}

\maketitle

\section{Introduction}

Consider an equilateral triangle of side $n$. Suppose that we cut out $n$ upward unit equilateral triangles, and attempt to tile the rest with unit $60^{\circ}$--$120^{\circ}$ rhombi formed by gluing two unit equilateral triangles together. This results in a \emph{fine mixed subdivision} of the triangle, sometimes also called a \emph{rhombus tiling of a holey triangle} (\cite{ArdBil06}).

\begin{figure}[h]
\centering
\includegraphics{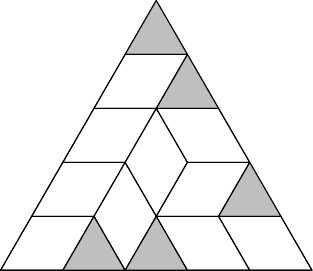}
\caption{A fine mixed subdivision of an equilateral triangle of side length 5.}\label{fig:ex}
\end{figure}

In \cite{ArdBil06}, Ardila and Billey described the condition of the positions of the holes for which the tiling is possible via a ``spread-out condition'' (see \cref{thm:spread}). In this paper, we focus on extending our understanding of the tilings with any given set of hole positions beyond their existence. Sections 3 to 6 explore the structure of the set of tilings for a fixed arrangement of triangular holes, and Section 7 discusses an operation that moves the holes and connects all possible tilings.

\section{Preliminaries}

For convenience, we will use ``upward'' and ``downward'' to describe the orientation of equilateral triangles. (The equilateral triangles in \cref{fig:ex} are all upwards.)

\begin{definition}[\cite{San03}, Definition 3.1]
    A \emph{fine mixed subdivision} of an upward equilateral triangle of side length $n$ is a partition into $n$ upward unit equilateral triangles (also referred to as \emph{holes} later) and $\binom{n}{2}$ unit-side-length rhombi formed by combining two unit equilateral triangles edge-to-edge. 
    
    (From now on, we will simply use ``rhombus'' to refer to such a rhombus. We will also drop the ``equilateral'' adjective for triangles as all triangles in this paper will be equilateral.)
\end{definition}

In the rest of the paper, we will sometimes also call such a fine mixed subdivision a \emph{rhombus tiling} of a \emph{holey triangle}, where the ``holey triangle'' is the region formed by removing the $n$ unit triangles from the large triangle. Such a rhombus tiling is naturally defined for any region comprised of several unit triangles on a triangular grid.

Suppose we're given a holey triangle and want to determine whether it has a rhombus tiling. Note that every upward triangle $T$ of side $k$ consists of $k$ more upward unit triangles than downward unit triangles. Also, note that a rhombus covers an equal number of upward and downward triangles in $T$ if it does not intersect the boundary of $T$, and covers one upward triangle if it intersects the boundary of $T$. Hence, there are at most $k$ triangular holes that lie in $T$. Ardila and Billey showed that if this condition is satisfied for every $T$, then the desired tiling exists.

\begin{definition}
    We will say that an arrangement of $n$ upward unit triangular holes in an upward triangle of side length $n$ is \emph{spread-out} if every upward grid-aligned triangle of size $k$ contains at most $k$ holes. We will also refer to the side length of such a triangle as its \emph{size}.
\end{definition}

\begin{theorem}[\cite{ArdBil06}, Theorem 6.2]\label{thm:spread}
    An arrangement of unit triangles is part of some fine mixed subdivision if and only if it's spread-out.
\end{theorem}

\begin{definition}
    For an arrangement of triangular holes, we say an upward triangle of side length $k$ is \emph{saturated} if it contains exactly $k$ triangular holes. 
\end{definition}

\begin{lemma}\label{lemma:saturated}
    If a triangle is saturated, then no rhombus can intersect its boundary in any tiling.
\end{lemma}

\begin{proof}
    Consider a saturated triangle $T$. A rhombus covers an equal number of upward and downward triangles in $T$ if it does not intersect the boundary of $T$, and covers one upward triangle if it intersects the boundary of $T$. Since the number of upward triangles in $T$ covered by a rhombus is equal to the number of downward triangles in $T$ covered by a rhombus, no rhombus can intersect the boundary of $T$.
\end{proof}

\section{Uniqueness of Tiling}

Suppose we are given an arrangement of unit triangular holes and would like to tile the rest with rhombi. If $T$ is a saturated triangle for the given arrangement, by \cref{lemma:saturated} this problem can be separated into two independent sub-problems: tiling $T$ and tiling the region outside $T$. Thus, it makes sense to consider a more general problem, where the triangular holes do not necessarily have unit side length.

More precisely, consider an upward triangle $\Delta$ of side length $n$ divided into $n^2$ unit triangles, and suppose that we cut out several disjoint upward triangular holes whose sizes sum to $n$.

\begin{figure}[h]
\centering
\includegraphics{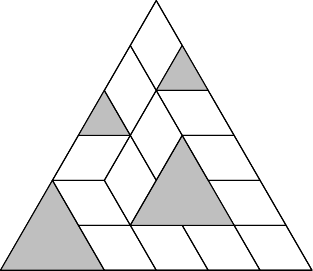}
\caption{A tiling with some holes whose side lengths are greater than 1.}
\end{figure}

We say a set of triangles $S$ with disjoint interiors is \emph{spread-out} if for any upward triangle $T$ inside $\Delta$,

\begin{equation*}
\sum_{s\in S} \text{size}(s\cap T) \le \text{size}(T).    
\end{equation*}

Call $T$ \emph{saturated} (by $S$) if equality holds. Note that all elements of $S$ are saturated triangles themselves.

Similarly to the case of unit triangular holes, one may prove that if $\Delta \setminus \bigcup_{s\in S} s$ is tile-able with rhombi, then $S$ is spread-out. If we replace all triangles in $S$ with upward unit triangles in their bottom rows, we get a spread-out arrangement of unit triangles, so $S$ being spread-out is sufficient for $\Delta \setminus \bigcup_{s\in S} s$ to be tile-able with rhombi, as a corollary of \cref{thm:spread}.

\iffalse
\begin{lemma} (Also see \cite{LiuYao22})
    For a spread-out arrangement $S$ in $\Delta$, if $T_1$ and $T_2$ are saturated and their intersection is nonempty, then $T_1 \cap T_2$ and $T_1 \vee T_2$ are saturated as well.
\end{lemma}

\begin{proof}
    By \cite{LiuYao22}, for any two triangles $K_1$ and $K_2$, we have 
    \begin{equation*}
        \text{size}(K_1) + \text{size}(K_2) \le \text{size}(K_1 \cap K_2) + \text{size}(K_1 \vee K_2),
    \end{equation*}
    with equality if and only if $K_1 \cap K_2$ is nonempty. Applying this claim for triangles $T_1 \cap s$ and $T_2 \cap s$ for any $s \in S$ gives
    \begin{equation*}
        \text{size}(T_1 \cap s) + \text{size}(T_2 \cap s) \le \text{size}((T_1 \cap T_2) \cap s) + \text{size}((T_1 \vee T_2) \cap s).
    \end{equation*}

    Summing this over all $s \in S$, we get
    \begin{equation*}
    \begin{split}
        \sum_{s\in S} \text{size}((T_1 \cap T_2) \cap s) + \sum_{s \in S} \text{size}((T_1 \vee T_2) \cap s) \ge \sum_{s\in S} \text{size}(T_1 \cap s) + \sum_{s\in S} \text{size}(T_2 \cap s) = \\
         \text{size}(T_1) + \text{size}(T_2) = \text{size}(T_1 \cap T_2) + \text{size}(T_1 \vee T_2),
    \end{split}
    \end{equation*}
    so both $T_1 \cap T_2$ and $T_1 \vee T_2$ must be saturated.
    
\end{proof}

\fi

For upward triangles $s_1$, $s_2$, denote by $s_1 \vee s_2$ the smallest upward triangle that contains both $s_1$ and $s_2$. Note that if saturated triangles $s_1, s_2 \in S$ intersect at exactly one point, then $s_1 \vee s_2$ is saturated, and there is only one way to tile $(s_1 \vee s_2) \setminus (s_1 \cup s_2)$. Also, no other triangles in $S$ can intersect $s_1\vee s_2$ in their interiors. Therefore, there is an obvious bijection between tilings of $\Delta$ with holes $S$ and tilings of $\Delta$ with holes $S \setminus \{s_1, s_2\} \cup (s_1 \vee s_2)$.

\begin{figure}[h]
\centering
\includegraphics{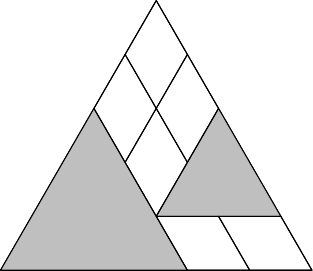}
\caption{Two touching saturated triangles $s_1$ and $s_2$ and the unique way to tile $(s_1 \vee s_2)\setminus (s_1\cup s_2)$.}
\end{figure}

Consider any set $S$ of triangular holes and a saturated triangle $T$. Let $A$ be the set of triangles in $S$ completely outside $T$, $B$ be the set of triangles in $S$ completely inside $T$, and $\{s_1, s_2, \dots, s_k\}$ be the rest of triangles in $S$.

For each $i$, let $p_i = T \cap s_i$, and pick $q_i$ such that $s_i = p_i \vee q_i$. Note that a tiling of $\Delta$ with holes $S$ corresponds to a tiling of $\Delta$ with holes holes $A \cup \{T\} \cup \{q_1, q_2, \dots q_k\}$ and a tiling of $T$ with holes $B \cup \{p_1, p_2, \dots p_k\}$.

Suppose that we start with spread-out arrangement of holes and begin performing changes of the following type: two triangles $s_1$, $s_2$ touching at a point can be replaced with $s_1\vee s_2$. Eventually, we will get an arrangement of triangular holes that are pairwise not touching. If this arrangement consists of a single hole coinciding with the big triangle, then the original arrangement of holes has a unique tiling. We claim that the converse is also true.

\iffalse
\begin{lemma}\label{lemma:slide}
   Let $S$ be a spread-out set of triangles, and let $t\in S$ be a triangle not touching the bottom boundary of $\Delta$. Let $t_1$ be the triangle obtained by sliding $t$ by one unit down and to the left. If $(S \setminus \{t\})\cup \{t_1\}$ is spread-out, then $\Delta \setminus \cup_{s\in S} s$ has a tiling that uses all rhombi adjacent to the bottom of $t$ and pointing to the left.
\end{lemma}

\begin{proof}
    Replace all triangles with all unit triangles in their bottom rows. Tilings with this triangle arrangement are clearly in bijection with tilings with the original triangle arrangement. We can first slide all triangles in the bottom row of $t$ down left by one unit (this preserves the spread-out condition) and then slide all the triangles to the bottom row.
\end{proof}
\fi 

\begin{theorem}\label{thm:nonunique}
    Let $S$ be a spread-out set of triangles.  If no two triangles in $S$ intersect, then  $\Delta \setminus \cup_{s\in S} s$ is tileable in at least two ways.
\end{theorem}

\begin{corollary}
    A spread-out set $S$ of $n$ unit triangles induces a unique tiling of a triangle of side length $n$ if and only if one can repeatedly replace two elements $s_1, s_2$ of $S$ such that $s_1\cap s_2$ is a single point with $s_1\vee s_2$ until there is only one element, which is the entire triangle.
\end{corollary}

\begin{remark}
    The number of uniquely tile-able spread-out arrangements for $n = 1, 2, \dots$ is \[1, 3, 16, 122, 1188, \dots.\] This is sequence A295928 on OEIS, which counts the number of triangular matrices $\{T_{i, j}\mid 1\leq i, j\leq n, i+j\leq n+1\}$ with exactly $n$ $1$-entries (and the rest being $0$'s) such that one can repeatedly replace triples of the form $(T_{i, j}, T_{i+1, j}, T_{i, j+1})$ with exactly two $1$'s with all $1$'s until the entire matrix only contain $1$'s. It is not difficult to see that this description is the same as our condition.
\end{remark}

We will delay the proof of \cref{thm:nonunique} until the end of the next section.

\begin{definition}[\cite{San03}, Section 1.1]
    A polyhedral subdivision is \emph{regular} if it can be realized as the projection of lower faces of a polytope.
\end{definition}

In the case of fine mixed subdivisions of dilated triangles, there is an equivalent condition of regularity that is more useful for our purpose.

\begin{lemma}[\cite{DevStr04}, Theorem 1, paraphrased]\label{lemma:real}
    A fine mixed subdivision of a triangle of side length $n$ is regular if and only if its corresponding tropical pseudoline arrangement can be realized as a tropical line arrangement.
\end{lemma}

Here we omit the definition of tropical (pseudo)lines, and instead present an example in \cref{fig:reg} for an intuitive understanding.

\begin{figure}[h]
    \centering
    \includegraphics{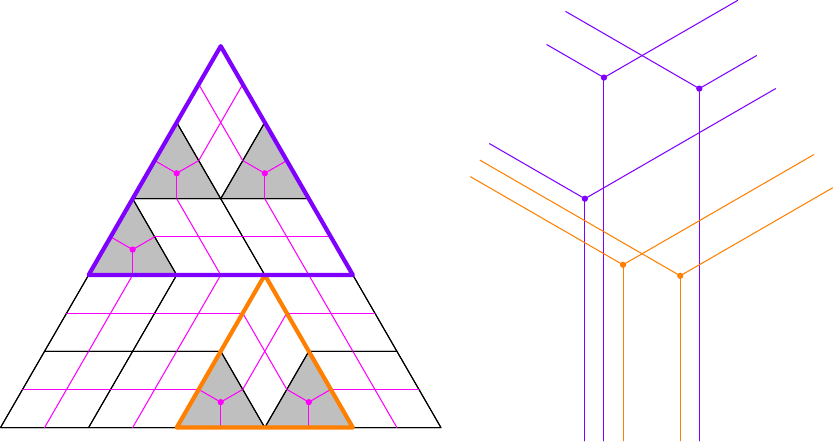}
    \caption{A fine mixed subdivision with its tropical pseudoline arrangement drawn in magenta, along with its realization (in purple and orange). The purple and orange triangles correspond to $s_1$ and $s_2$ in the proof, respectively.} \label{fig:reg}
\end{figure}

In \cite{ArdBil06}, the authors additionally conjectured that any spread-out arrangement of unit triangles can be completed into a \emph{regular} fine mixed subdivision. This has been disproved by Yoo with a counterexample in $n=6$, where there are two possible fine mixed subdivisions but neither is regular (see Section 8.1 of \cite{ArdCeb13}). We argue that this is in a sense the ``minimal'' counterexample.

\begin{proposition}
    If a spread-out arrangement of unit triangles induces a unique fine mixed subdivision, then this subdivision is regular.
\end{proposition}

\begin{proof}[Proof sketch]
    Based on the uniqueness condition, it suffices to show that if two triangles $s_1$ and $s_2$ touch and each has a regular subdivision, then $s_1\vee s_2$ also has a regular subdivision. 

    From \cref{lemma:real} we can assume that $s_1$ and $s_2$ each correspond to some tropical line arrangement $T_1$ and $T_2$. Assume without loss of generality that $s_2$'s top vertex is touching $s_1$'s bottom edge, as shown in \cref{fig:reg}. In this case, we can shrink $T_2$ and place all of its vertices between the two downward rays of $T_1$ that correspond to the two sides of the point $s_1\cap s_2$, and sufficiently far below so that the leftward and rightward arms of $T_2$ are all below $T_1$'s. It is not difficult to see that this gives a realization of the unique tiling of $s_1\vee s_2$.
\end{proof}

\section{A Graph Representation of Fine Mixed Subdivision}

Consider some fine mixed subdivison. Note that it divides the plane into regions: triangular holes, rhombi, and the exterior region. Consider the graph $G$ whose vertices are these regions, and where there is an edge $a \rightarrow b$ if and only if $a$ contains an upward triangle adjacent to a downward triangle contained by $b$, or $a$ is a region adjacent to the boundary and $b$ is the exterior region. We say that $G$ is the graph corresponding to the original fine mixed subdivision.

\begin{figure}[h]
\centering
\includegraphics{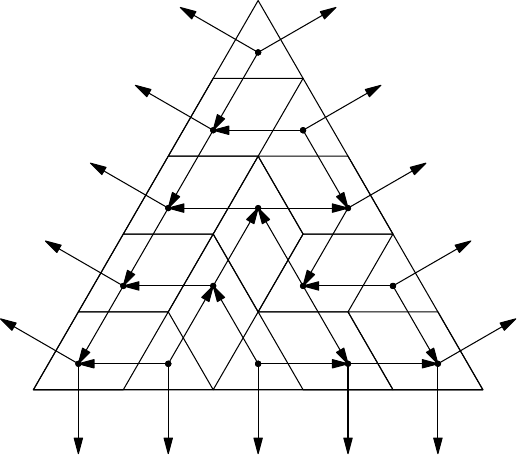}
\caption{A fine mixed subdivision and the corresponding graph.}
\end{figure}

Note that that all vertices corresponding to rhombi have indegree and outdegree equal to $2$; each vertex corresponding to a triangular hole of side $k$ has outdegree $3k$ and indegree $0$; and the vertex corresponding to the outside region has outdegree $0$ and indegree $3n$.

Consider any cycle of the graph. This cycle must consist of rhombi that occupy a sequence of triangles where each next one touches the previous and the last one touches the first one. Then this sequence of triangles can be tiled in another way. We will call replacing one of these tilings with the other one a ``cycle flip''.

\begin{figure}[h]
\centering
\includegraphics{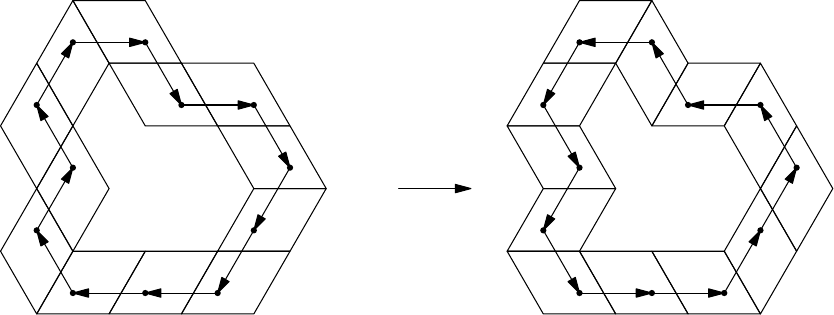}
\caption{A cycle flip.}
\end{figure}

\begin{lemma}\label{lemma:cycle}
    Consider any fine mixed subdivision $\mathcal{F}$, and let $a$ and $b$ be two rhombi such that $a \rightarrow b$ is an edge of its corresponding graph. Let $c$ be the rhombus formed by the upward triangle of $a$ and the downward triangle of $b$. Then there exists a fine mixed subdivision with the same arrangement of triangular holes as $\mathcal{F}$ and using $c$ if and only if $a$ is reachable from $b$ in the graph corresponding to $\mathcal{F}$.
\end{lemma}

\begin{proof}
    If $a$ is reachable from $b$, then there is a simple cycling containing the edge $a \rightarrow b$. After flipping this cycle, we get the desired fine mixed subdivision.

    Suppose that there is a fine mixed subdivision $\mathcal{G}$ using $c$. Build a sequence of rhombi as follows: let $f_1 = a$. Let $g_1$ be the rhombus in $\mathcal{G}$ containing the downward triangle of $f_1$. Let $f_2$ be the rhombus in $\mathcal{F}$ containing the upward triangle of $g_1$. Let $g_2$ be the rhombus in $\mathcal{G}$ containing the downward triangle of $f_2$, and so on. Since $f_i$ uniquely determines $f_{i-1}$ and $f_{i+1}$, there is an index $k > 1$ such that $f_k = f_1$. Then vertices $a = f_1, f_2, f_3, \dots f_{k-1} = b$ form an cycle in the graph corresponding to $\mathcal{F}$, as desired.
\end{proof}

This gives us another criterion for the uniqueness of a tiling: any tiling is unique for its arrangement of triangles if and only if the corresponding graph has no oriented cycles.

The notion of a cycle flip also gives us a simple way to prove \cref{thm:nonunique}.

\begin{proof}[Proof of \cref{thm:nonunique}]
    Consider a spread-out arrangement where no two trianglular holes touch, and any rhombus $a$ of some tiling of this holey triangle. We claim that there is at least one rhombus $b$ such that the graph contains the edge $a \leftarrow b$. Indeed, each of the two edges going into $a$ is coming from either a triangle or a rhombus; if both come from a triangle then the two triangles are touching. Thus, we can construct an infinite sequence of edges $a_1 \leftarrow a_2 \leftarrow a_3 \leftarrow \dots$, so there must be a cycle in the graph, and we can perform a cycle flip, giving us a second tiling.
\end{proof}

\section{GD Flips}

Using the the idea from the proof of \cref{lemma:cycle}, we see that any two fine mixed subdivisions with the same arrangement of triangular holes differ by several cycle flips. Additionally, if we place $3$ unit triangles in the corners and one triangle of side $n-3$ in the middle we get an arrangement of triangular holes that admits only two fine mixed subdivisions, where only one cycle flip is possible. (See \cref{fig:twotilings}.)

\begin{figure}[h]
\centering
\includegraphics{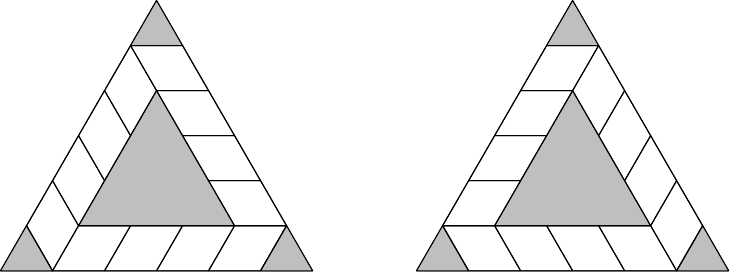}
\caption{An arrangement of triangles admitting exactly two fine mixed subdivisions for $n=6$.}\label{fig:twotilings}
\end{figure}

\begin{definition}
    For $n\ge 3$, consider the region obtained by removing three corner unit triangles from a triangle of side $n$, as well as removing a triangle of side $n-3$ in the middle. There are two configurations of rhombi that tile this region. Let the \emph{Clockwise GD} (CW-GD) of size $n-3$ be the configuration where vertical rhombi (the ones without horizontal edges) go along the left side of the triangle, and let \emph{Counter-clockwise GD} (CCW-GD) be the other one.\footnote{This name is inspired by the (pre-2020) Google Drive logo.} A \emph{GD flip} consists of replacing a CW-GD with a CCW-GD. In particular, a GD flip is a cycle flip.
\end{definition}

\begin{figure}[h]
\centering
\includegraphics{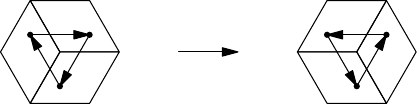}

\bigskip

\includegraphics{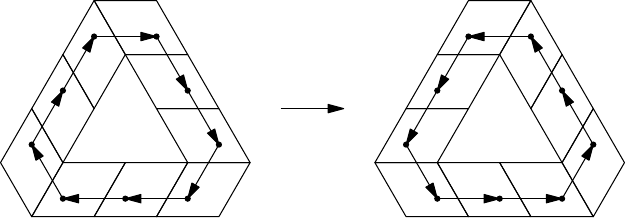}
\caption{GD flips of size 0 and 2.}
\end{figure}

It turns out that these types of flips are the only ones we need to connect all tilings with the same arrangement of holes.

\begin{theorem}\label{thm:gdflips}
   Any two tilings with the same arrangement of triangles are connected by a sequence of GD flips and their inverses.
\end{theorem}

\begin{lemma}\label{lemma:cwgd}
    A tiling has a clockwise cycle if and only if it has a CW-GD.
\end{lemma}

\begin{proof}
    Take a minimal clockwise cycle $C$ (in terms of enclosed area). 
    
    First, we claim that there are no edges going from a rhombus strictly inside $C$ to a rhombus of $C$. Suppose there is such an edge $a \rightarrow b$. Without loss of generality, we may assume that the boundary between $a$ and $b$ is horizontal, that is, $a$ is directly below $b$. Then we can construct a sequence of rhombi starting with $b,$ $a$ where each next rhombus is directly below the previous one. Eventually, this sequence must hit $C$ again. Thus, we've constructed a path surrounded by $C$ connecting two vertices of $C$. This path creates a clockwise cycle smaller than $C$, which is a contradiction.

    Now consider the polygon surrounded by $C$. Note that if we traverse it clockwise, its edges can only go in three directions, corresponding to traversing an upward triangle in clockwise direction. This means that all internal angles of this polygon are equal to $60^\circ$ or $300^\circ$. If some is equal to $300^\circ$, then only one of the triangles surrounding can belong to $C$, which is impossible. Since all angles of the polygon are equal to $60^\circ$, it has to be a triangle, so $C$ is a CW-GD, as desired.
\end{proof}

\begin{figure}
\centering
\includegraphics{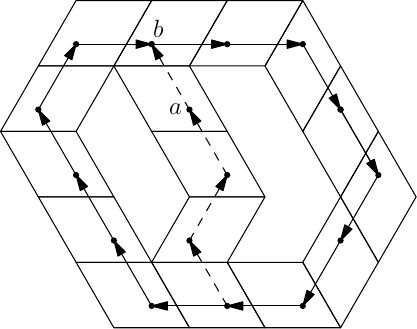}
\caption{A clockwise cycle that contains a smaller clockwise cycle.}
\end{figure}

\begin{proof}[Proof of \cref{thm:gdflips}]
Suppose that we begin with some tiling and start performing GD flips while possible. Since the number of vertical rhombi does not change after a GD flip, and the sum of $x$-coordinates of their centers strictly increases with every GD flip, we must eventually stop after finitely many flips. Note that there is at most one tiling with no clockwise cycles: if there were two, they would differ by several cycle flips, so one of them would have a clockwise cycle. This shows that starting from any tiling, we can get to the unique tiling with no clockwise cycles by performing GD flips, and then back to any other tiling by performing their inverses.
\end{proof}

We can construct the unique tiling with no clockwise cycles more explicitly:

\begin{proposition}
    If in the process of sliding triangles (see \cite{LiuYao22}, Section 3) we slide the triangle left whenever possible, we end up with a tiling that has no CW-GD.
\end{proposition}

\begin{proof}
    Suppose there is a CW-GD. Then its topmost rhombus arose from us sliding a triangle to the right when we could slide it to the left.
\end{proof}

\section{Segments Forced by Holes}

Given a fixed arrangement of triangular holes, it is natural to ask if there is a tiling of the remaining region including a certain rhombus. This is equivalent to determining if a certain segment of the triangular grid is forced to be a part of the tiling. Note that \cref{lemma:saturated} shows that all segments belonging to a boundary of a saturated triangle are forced, and \cref{lemma:cycle} shows that a segment between two rhombi is forced if and only if they belong to different connected components. 

Given a fined mixed subdivision with corresponding graph $G$, consider the subgraph of $G$ induced by vertices corresponding to rhombi. In this graph, pick a strongly connected component that has no edges going in from other strongly connected components. Let $A$ be the set of rhombi in this component, as well as all saturated triangles adjacent to them. Note that in $G$, there are no edges going from a vertex outside $A$ into a vertex in $A$.

Consider the union of regions corresponding to vertices of $A$, which clearly is connected. Consider the polygon $P$ that is the outside boundary of this region. This polygon, when traversed clockwise, can only have edges going in the directions corresponding to traversing an upward triangle in clockwise direction. It is also not difficult to see that all outward-pointing corners of $P$ must be part of saturated triangles, or else the rhombi occupying those corners cannot be part of the initial strongly connected component.

\begin{figure}[h]
\centering
\includegraphics{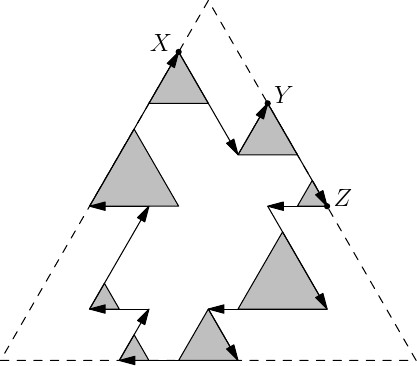}
\caption{A polygon $P$ traversed clockwise and the corresponding triangle $T$ (dashed).}\label{fig:forced}
\end{figure}

Let $T$ be the smallest upward triangle containing $P$. Note that if $X$ and $Y$ are the topmost vertices of $P$ on left and right sides of $T$, respectively, then the distance from $X$ to $Y$ along the boundary of $P$ is not less than the distance from $X$ to $Y$ along the boundary of $T$ (in both cases, we are traversing polygons clockwise). Indeed, every unit segment on the path from $X$ to $Y$ along the boundary of $P$ increases the $x$-coordinate by at most $\frac 12$, and each unit segment on the path from $X$ to $Y$ along the boundary of $T$ increases the $x$-coordinate by exactly $\frac 12$. Similarly, if $Z$ is the bottommost vertex of $P$ on the right side of $T$, the distance from $Y$ to $Z$ along the boundary of $P$ is not less that the distance from $Y$ to $Z$ along the boundary of $T$. This shows that $\text{perimeter}(P) \ge \text{perimeter}(T)$.

However, note that every unit segment on the boundary of $P$ corresponds to an edge going from a vertex in $A$ to a vertex not in $A$, and there are no edges going from a vertex not in $A$ into a vertex in $A$, so the total outdegree of vertices of $A$ larger than the total indegree of vertices in $A$ by at least $\text{perimeter}(P)$. This means that the sum of sizes of triangular holes in $A$ is at least $\text{perimeter}(P)/3 \ge \text{size}(T)$. Since all these triangular holes lie entirely inside $T$, all inequalities must turn into equalities, and $T$ is saturated.

In particular, all segments on the boundary of $P$ between $X$ and $Y$ must go down-right or up-right, and all segments on the boundary of $P$ between $Y$ and $Z$ must go down-right. This implies that $T \setminus P$ consists of three regions containing three vertices of $T$, and there is only one way to tile these regions.

This also implies that $A$ is simply connected -- if $A$ had ``holes'', then there would be edges going from vertices of $A$ to vertices in the holes, so the difference of total out- and in-degrees of vertices of $A$ would be strictly greater than $\text{perimeter}(P)$.

Now, we may replace $T$ with a triangular hole and repeat the process. This results in the following categorization of the forced segments:

\begin{theorem}
     Given a set of triangular holes, the boundaries of saturated triangles are forced, and remaining forced segments come from repeated placing rhombi in the three corners of saturated triangles until obstructed by holes.
\end{theorem}

For example, in \cref{fig:forced}, one can place rhombi in the regions $T\setminus P$, and the obstruction come from the fact that the outward-pointing corners of $P$ are all holes. This also gives a poly-time algorithm for determining the forced segments, even though the number of possible tilings may be exponential.

\section{Trapezoid Flips}

\begin{definition}
    A \emph{trapezoid flip} consists of replacing a triangle and an adjacent rhombus with a different triangle and an adjacent rhombus occupying the same trapezoid.
\end{definition}

\begin{figure}[h]
\centering
\includegraphics{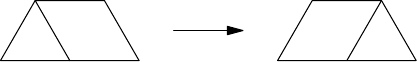}
\caption{A trapezoid flip.}
\end{figure}

It is known that \emph{all} fine mixed subdivisions of a dilated triangle can be connected by such flips.

\begin{proposition}[\cite{San03}, Corollary 4.5] 
    Any two fine mixed subdivisions are connected by at most $\frac {2n(n-1)}{3}$ trapezoid flips.
\end{proposition}

\begin{proof}
    For any triangle hole $s$, let $h_1(s),$ $h_2(s)$, and $h_3(s)$ denote the distances from sides $s$ to the parallel sides of the big triangle. Note that for any triangle $s$ with $h_i(s) > 0$, there is a trapezoid flip that decreases $h_i(s)$ by $1$. Then the distance from a fine mixed subdivision with triangle set $S$ to the unique fine mixed subdivision with all triangles at the bottom is at most $\sum_{s \in S} h_i(s).$ Then the distance between two fine mixed subdivisions with triangle sets $S_1$ and $S_2$ is at most $\sum_{s \in S_1} h_i(s) + \sum_{s \in S_2} h_i(s).$ Since for any $s$ we have $h_1(s) + h_2(s) + h_3(s) = n - 1,$ the sum of this expression over $i = 1, 2, 3$ is equal to $2n(n-1)$, so for some choice of $i$ its at most $\frac{2n(n-1)}{3}$, as desired.
    
\end{proof}

We show that this bound is tight up to terms of order $O(n)$. 

\begin{proposition}
    Place a triangle of side $n$ over a grid of unit hexagons. Consider two fine mixed subdivisions, where one is obtained by tiling each hexagon in clockwise manner and the other obtained by tiling each hexagon in counterclockwise manner (trimming off half of each rhombus that intersects the boundary of the triangle). At least $\frac{2n^2}{3} - n - O(1)$ trapezoid flips are needed to change from one tiling to the other.
\end{proposition}

\begin{figure}[h]
\centering
\includegraphics{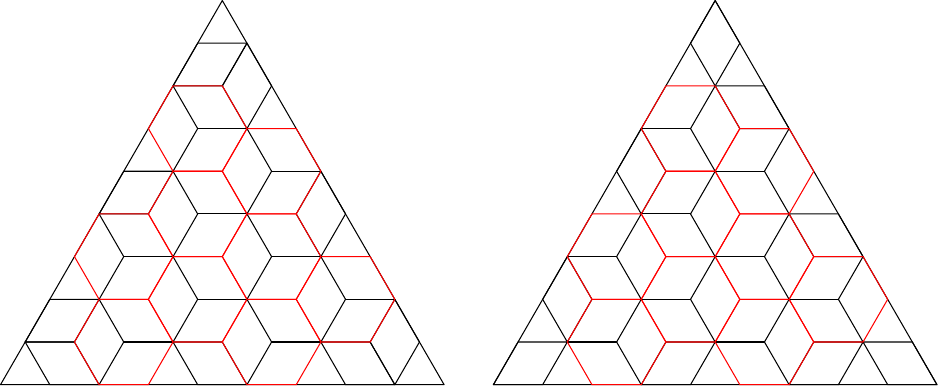}
\caption{Two tilings that require many trapezoid flips to connect, with common hexagon boundaries highlighted in red.}
\end{figure}

\begin{proof}
    For each rhombus in the final tiling, consider the last trapezoid flip involving this rhombus. Since these flips are all distinct and the two tilings have no rhombi in common, this gives us $\frac{n^2-n}{2}$ distinct trapezoid flips.

    For each hexagon, consider the first trapezoid flip involving some rhombus inside this hexagon in the initial tiling. These flips are clearly pairwise distinct. Additionally, they result in a rhombus that does not belong to any hexagon, so these moves are distinct from the previously considered $\frac{n^2-n}{2}$. Note that there are $\frac{n^2}{6} - \frac{n}{2} - O(1)$ hexagons lying inside the triangle. Therefore, we have $\frac{n^2-n}{2} + \frac{n^2}{6} - \frac{n}{2} - O(1) = \frac{2n^2}{3} - n - O(1)$ distinct trapezoid flips, as desired.
\end{proof}

\section{Further Questions}

The rhombus tilings of a hexagon with sides $a, b, c, a, b, c$ are in bijection with $3$-dimensional Young tableaux that fit in a $a \times b \times c$ box. A GD flip of size $0$ in this case is also known as a \emph{hexagon flip}, and corresponds to removing a cube.

\begin{figure}[h]
\centering
\includegraphics{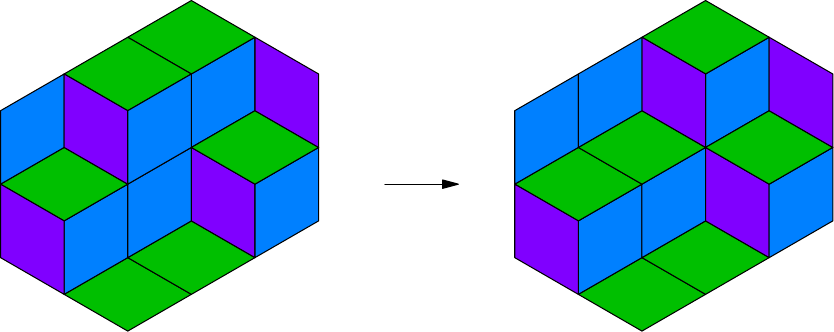}
\caption{A GD flip of size $0$ (note that this is rotated by $60^\circ$ counterclockwise relative to conventions in the rest of the paper for ease of visualization).}
\end{figure}

Tilings of regions that may contain triangular holes of some fixed orientation correspond to similar structures in certain $3$-dimensional manifolds. It might be interesting to find an analog of Young tableaux that can help with understanding fine mixed subdivisions of triangles.

\begin{figure}[h]
\centering
\includegraphics{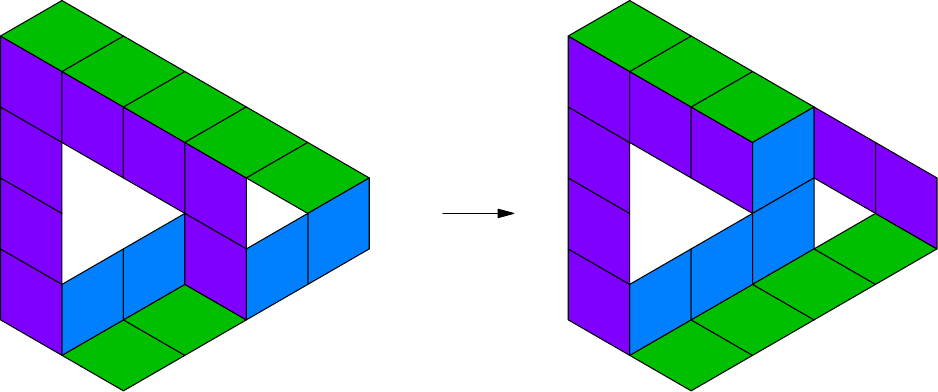}
\caption{A GD flip of nonzero size appears to correspond to removing some non-Euclidean polyhedron.}
\end{figure}

One attempt of studying this correspondence is by looking at what we called the \emph{depth function}. Fix a direction parallel to one of the lines of the triangle grid. For any node $p$ of the triangular grid, consider the ray starting at $p$ and going in the chosen direction, and let $d(p)$ denote the number of rhombi whose short diagonals belong to that ray. Note that the values of $d$ at neighboring nodes differ by at most one, and that a GD flip changes the values of $d$ of nodes in the triangle surrounded by the flip by one, and does not change all other values of $d$.

Note that the unique tiling with no clockwise cycles corresponds to a structure where no polyhedron can be removed, i.e., the ``empty'' one. Similarly, the unique tiling with no counterclockwise cycles corresponds to a structure where no polyhedron can be added, i.e., the ``full'' one. For each node, the difference of values of $d$ for full and empty configurations corresponds to how ``thick'' the region that we're filling with polyhedra is at that point.

\end{document}